\documentclass[12pt]{amsart}
\usepackage{amssymb}
\usepackage{amsmath}
\makeatletter
\renewcommand*\env@matrix[1][*\c@MaxMatrixCols c]{%
  \hskip -\arraycolsep
  \let\@ifnextchar\new@ifnextchar
  \array{#1}}
\makeatother
\usepackage{amsthm}

\usepackage{xcolor}
\usepackage{epsfig}

\newtheorem{dfn}{Definition}[section]
\newtheorem{thrm}[dfn]{Theorem}

\newtheorem{lm}[dfn]{Lemma}

\newtheorem{cor}[dfn]{Corollary}

\newcommand{\rem}{\noindent{\bf Remark.  }}

\numberwithin{equation}{section}

\title
[An Ambarzumian type theorem on graphs]{An Ambarzumian type theorem on graphs with odd cycles}

\author{M\'arton Kiss}
\thanks{This work was supported by the Hungarian NKFIH Grant SNN-125119.\newline
2010. Math. Subject Classification: Primary 34A55, 34B20, 34B24, 34B45; Secondary 34L40, 47A75\newline
}
\address
{Department of Differential Equations
\newline
\indent Institute of Mathematics
\newline
\indent Budapest University of Technology and Economics
\newline
\indent H 1111 Budapest, M\H{u}egyetem rkp. 3-9.}
\email{mkiss@math.bme.hu}
\begin{document}

\begin{abstract}
We consider an inverse problem for Schr\"odinger operators on a connected equilateral graph $G$ with standard matching conditions. The graph $G$ consists of at least two odd cycles glued together at a common vertex. We prove an Ambarzumian type result, i.e., if a specific part of the spectrum is the same as in the case of zero potential, then the potential has to be zero. 
\end{abstract}
\maketitle
\begin{section}{Introduction}

The aim of this paper is to make a statement in \cite{Yang20161348} more accurate and general. %
The addressed problem originates from a work of Ambarzumian \cite{A} on reconstruction of a differential operator from its eigenvalues. The material from the theory of Sturm-Liouville equations is summarized in \cite{B1, Horvath2005Annals, ChengWangWu2010, Horvath2015}, previous results for graphs are \cite{Pivovarchik2005ambarzumian, CarlsonPivovarchik2007ambarzumian, LawYanagida2012, Davies2013}. Another source of the problem is the so-called \emph{quantum graphs}, i.e., differential operators on graphs \cite{BerkolaikoKuchment2013book, PokornyiBorovskikh2004, Kuchment2008}). A third component is the calculation of \emph{spectral determinant}s (or alternatively \emph{functional determinant}s or \emph{characteristic function}s) \cite{CurrieWatson2005, KacPivovarchik2011, CarlsonPivovarchik2008SpectralAsymptotics, Pankrashkin2006, AkkermansComtetDesboisMontambauxTexier2000, Desbois2000, Friedlander2006, HarrisonKirstenTexier2012, Texier2010, MollerPivovarchik2015book}. For a more detailed discussion of these results see the Introduction in \cite{Kiss2016Ambarzumian}.  

\end{section}
\begin{section}{Results and discussion}

Let $r\ge2$ and consider $r$ cycle graphs $C_1,C_2,\ldots,C_r$ with odd cycle lengths $n_1,n_2,\ldots,n_r$ ($n_j=1$ is also possible). Let the vertices of $C_j$ be $v_{j0},\ldots,v_{j{n_j}}=v_{j0}$, and let us form the graph $G$ as the union of $C_j$'s, identifying the vertex $v_{j0}$ for all $j$. We shall say that $G$ is a graph consisting of $r\ge2$ odd cycles glued together at a common vertex. The edge of $G$ between $v_{j\,k-1}$ and $v_{jk}$ is sometimes denoted by $e_{jk}$; however, when the particular location of the edges are not important, we shall refer to them as $e_1,e_2,\ldots,e_{|E|}$.

Choosing an arbitrary orientation, we parametrize each edge with $x\in[0,1]$, and consider a Schr\"odinger operator with potential $q_{j}(x)\in L^1(0,1)$ on the edge $e_{j}$ and with Neumann (or Kirchhoff) boundary conditions (sometimes called standard matching conditions), i.e., solutions are required to be continuous at the vertices and, in the local coordinate pointing outward, the sum of derivatives is zero. 
More formally, consider the eigenvalue problem
\begin{align}\label{sch}
-y''+q_j(x)y=\lambda y
\end{align}
on $e_j$ for all $j$ with the conditions
\begin{align}\label{continuity}
y_j(\kappa_j)=y_k(\kappa_k)
\end{align}
if $e_j$ and $e_k$ are incident edges attached to a vertex $v$ where $\kappa=0$ for outgoing edges, $\kappa=1$ for incoming edges (and can be both $0$ or $1$ for loops); and in every vertex $v$
\begin{align}\label{Kirchhoff}
\sum_{e_j\textrm{ leaves }v} y_j'(0)=\sum_{e_j\textrm{ enters }v} y_j'(1)
\end{align}
(loops are counted on both sides).

\begin{thrm}\label{oddcycles}
Consider the eigenvalue problem (\ref{sch})-(\ref{Kirchhoff}). Let $G$ be a graph consisting of $r\ge2$ odd cycles glued together at a common vertex. If $\lambda=0$ is the smallest eigenvalue and for infinitely many $k\in\mathbb{Z}^+$ there are $r-1$ eigenvalues (counting multiplicities) such that $\lambda=(2k+1)^2\pi^2+o(1)$, then $q=0$ a.e. on $G$.
\end{thrm}

If the lengths of the odd cycles are all $1$, i.e., the cycles are all loops, then the statement reduces to that of Theorem 2.1 in \cite{Yang20161348}, which states the following:

Suppose $G$ is a flower-like graph, i.e., a single vertex attached $r$ loops of length $1$. For $k=1,2,\ldots$, let $m_k$ be a sequence of integers with $\lim m_k=+\infty$. If eigenvalues are nonnegative, $\lambda_k=(2m_k+1)^2\pi^2$ are eigenvalues with multiplicities $(r-1)$, where $m_k$ is a strictly ascending infinite sequence of positive integers, then $q_j(x)=0$ a.e. on $[0,1]$, for each $j=1,2,\ldots,r$. We have to require $r\ge2$ for the consequence to hold.

\end{section}

\begin{section}{Calculation of the spectral determinant}

Denote by $c_j(x,\lambda)$ the solution of (\ref{sch}) which satisfies the conditions
$c_j(0,\lambda)-1 =c_j'(0,\lambda) = 0$ and by $s_j(x,\lambda)$ the solution of (\ref{sch}) which satisfies the conditions $s_j(0,\lambda) =s_j'(0,\lambda)-1 = 0$. %
Each $y_j(x,\lambda)$ may be written as a linear combination
\begin{align}
y_j(x,\lambda)=A_j(\lambda)c_j(x,\lambda)+\sqrt{\lambda}B_j(\lambda)s_j(x,\lambda).
\end{align}
Then $y_j(0,\lambda)=A_j(\lambda)$ is the same on each outgoing edge; hence as in \cite{Kiss2016Ambarzumian}, we index the functions $A(\lambda)$ by vertices, and then
\begin{align}%
y_j(x,\lambda)=A_v(\lambda)c_j(x,\lambda)+\sqrt{\lambda}B_j(\lambda)s_j(x,\lambda),
\end{align}
if $e_j$ starts from $v$. If the eigefunctions are normalized, i.e., $\sum_j\|y_j(x,\lambda)\|_2^2=1$, then $A_v(\lambda)=B_j(\lambda)=O(1)$ (\cite{CarlsonPivovarchik2007ambarzumian,Yang20161348,Kiss2016Ambarzumian}). The coefficients $A_v$ and $B_j$ form a $(|V|+|E|)$-dimensional vector, which satisfies $|V|$ Kirchhoff conditions at the vertices and $|E|$ continuity conditions at the incoming ends of edges, namely, for all $v\in V(G)$,
\begin{align}%
\sum_{e_j:\ldots\to v} \frac{1}{\sqrt{\lambda}}A_{v_j}(\lambda)c'_j(1,\lambda)+B_j(\lambda)s_j'(1,\lambda)-\sum_{e_j:v\to\ldots}B_j(\lambda)=0,
\end{align}
where in the first sum $v_j$ denotes the starting point of $e_j$; and for all $e_j\in E(G)$,
\begin{align}%
A_u(\lambda)c_j(1,\lambda)+\sqrt{\lambda}B_j(\lambda)s_j(1,\lambda)-A_v(\lambda)&=0,
\end{align}
if $e_j$ points from $u$ to $v$ (see eq. (2.3) and (2.4) in \cite{Kiss2016Ambarzumian}).

The matrix $M$ of this homogeneous linear system of equations has a special structure, the description of which we repeat from \cite{Kiss2016Ambarzumian}. Namely, $M=\left[\begin{array}{cc}A&B\\C&D\end{array}\right]$, where
\begin{itemize}
\item $A$ is like an adjacency matrix; $a_{vu}=\frac{1}{\sqrt{\lambda}}\sum c_j'(1,\lambda)$, the sum is taken on edges pointing from $u$ to $v$;
\item $B$ and $C$ are like incidence matrices;\\
$b_{vj}=\left\{\begin{array}{ccc}s_j'(1,\lambda)&\textrm{ if $e_j$ ends in $v$}\\-1&\textrm{ if $e_j$ starts from $v$}\\s_j'(1,\lambda)-1&\textrm{ if $e_j$ is a loop in $v$}\\0&\textrm{otherwise}\end{array}\right.$\\
$c_{jv}=\left\{\begin{array}{ccc}-1&\textrm{ if $e_j$ ends in $v$}\\c_j(1,\lambda)&\textrm{ if $e_j$ starts from $v$}\\-1+c_j(1,\lambda)&\textrm{ if $e_j$ is a loop in $v$}\\0&\textrm{otherwise}\end{array}\right.$
\item $D$ is a diagonal matrix, $d_{jj}=\sqrt{\lambda}s_j(1,\lambda)$.
\end{itemize}

The determinant of the matrix $M$ is the so-called spectral determinant of the problem (\ref{sch})-(\ref{Kirchhoff}).

\noindent{\bf Example 1. }
Consider a flower-like graph, i.e., a single vertex with $r$ loops. Then
\begin{equation*}
M=M_1=\left[\begin{array}{c|ccc}
  \frac{1}{\sqrt{\lambda}}\sum_{k=1}^{r}c_{k}'(1,\lambda)&s_{1}'(1,\lambda)-1&\ldots&s_{r}'(1,\lambda)-1\\\hline
  -1+c_{1}(1,\lambda)&\sqrt{\lambda}s_{1}(1,\lambda)&\ldots&0\\
  \vdots&\vdots&\ddots&\vdots\\
  -1+c_{r}(1,\lambda)&0&\ldots&\sqrt{\lambda}s_{r}(1,\lambda)
\end{array}\right].
\end{equation*}
with determinant
\begin{align*}
\det M_1&={\lambda}^{\frac{r-1}{2}}\left(\sum_{k=1}^{r}c_{k}'(1,\lambda)\prod_{j=1}^{r}s_{j}(1,\lambda)\right.\\
&\left.-\sum_{k=1}^{r}(s_{k}'(1,\lambda)-1)(-1+c_{k}(1,\lambda))\prod_{j\ne k}s_{j}(1,\lambda)\right)
\end{align*}
corresponding to formula (2.9) in \cite{Yang20161348}.

The elements of $M$ have the following asymptotics for $\lambda=(2k+1)^2\pi^2+d+o(1)$ (see \cite{CarlsonPivovarchik2007ambarzumian} eq. (2.3) or \cite{LawYanagida2012} Lemma 3.1):
\begin{align}
\frac{1}{\sqrt{\lambda}}c_j'(1,\lambda)&=\frac{1}{2\sqrt{\lambda}}(d-\int_0^{1}q_j)+o(\frac{1}{\sqrt{\lambda}}),\label{asym1}\\
s_j'(1,\lambda)&=-1+o(\frac{1}{\sqrt{\lambda}}),\label{asym2}\\
c_j(1,\lambda)&=-1+o(\frac{1}{\sqrt{\lambda}}),\label{asym3}\\
\sqrt{\lambda}s_j(1,\lambda)&=\frac{1}{2\sqrt{\lambda}}(\int_0^{1}q_j-d)+o(\frac{1}{\sqrt{\lambda}}).\label{asym4}
\end{align}

\rem
Using these asymptotics, we get
\begin{align}
\det M_1=-4\sum_{k=1}^{r}\prod_{j\ne k}\sqrt{\lambda}s_{j}(1,\lambda)+o(\lambda^{-\frac{r}{2}}).
\end{align}
This is a special case of (\ref{detM}) and of (\ref{detMforspecialG}) below.

\end{section}
\begin{section}{The proof}

\begin{lm}
The determinant of $M$ for $\lambda=(2k+1)^2\pi^2+O(1)$ is $O(\lambda^{-\frac{1}{2}(|E|-|V|)})$.
\end{lm}
\begin{proof}
Look at the terms in the Laplace expansion.
Taking at most $|V|$ factors from $B$ (and consequently from $C$) we have to take at least $(|E|-|V|)$ factors of magnitude $O(\frac{1}{\sqrt{\lambda}})$ from $A$ and $D$.
\end{proof}

We shall use the following terminology: a graph is a \emph{saturated forest} if every component has exactly one cycle. The saturated forest is \emph{odd}, if it does not contain even cycles.
Note that in a saturated forest, the number of edges and vertices are equal.

\begin{lm}\label{satforest}
The determinant of the (unoriented) incidence matrix of an odd saturated forest is $\pm 2^{\kappa}$ where $\kappa$ denotes the number of components.
\end{lm}
\begin{proof}
The incidence matrix (see in the Glossary) is a direct sum of that of the components, hence it is enough to prove the statement for connected graphs. It is true for odd cycles as well as for a single vertex with a loop. If the graph is not a single cycle or a single loop, there is at least one vertex with only one incident edge. Leaving this vertex out from the graph does not change the absolute value of the determinant of the incidence matrix. This can be repeated until we reach a single cycle or a single loop.
\end{proof}
\rem
If a graph contains even cycles, the determinant of its incidence matrix is $0$. 

\begin{lm}\label{incidence}
If $\lambda=(2k+1)^2\pi^2+O(1)$, then the determinant of a $|V|\times |V|$ submatrix of $C$ (and of $B$) is $\pm 2^{\kappa}+o(\frac{1}{\sqrt{\lambda}})$ if the rows in $C$ (columns in $B$) corresponds to the edges of an odd saturated forest, and $o(\frac{1}{\sqrt{\lambda}})$ otherwise.
\end{lm}
\begin{proof}
Leaving out the $o(\frac{1}{\sqrt{\lambda}})$ terms from the submatrix we make only $o(\frac{1}{\sqrt{\lambda}})$ error in its determinant. What we get is the negative of an incidence matrix of a subgraph with $|V|$ vertices and $|V|$ edges. If the subgraph contains an even cycle, then the corresponding rows are dependent. If a component of the subgraph contains more than one odd cycles or loops (and no even cycles), then consider a path between the two; the corresponding rows are also dependent. The number of cycles is equal to the number of components hence if the determinant is not $o(\frac{1}{\sqrt{\lambda}})$, the rows in $C$ must correspond to the edges of an odd saturated forest. The proof for $B$ is similar.
\end{proof}

\begin{thrm}\label{direct}
If $\lambda=(2k+1)^2\pi^2+O(1)$, then
\begin{align}\label{detM}
\det M=(-1)^{|V|}\sum_{\tau}4^{\kappa(\tau)}\prod_{e_j\notin\tau}\sqrt{\lambda}s_j(1,\lambda)+o(\lambda^{-\frac{|E|-|V|+1}{2}}).
\end{align}
where the sum is taken for all odd saturated forest $\tau$ of $G$.
\end{thrm}
\begin{proof}
The main terms in the Laplace expansion are those which contain exactly $(|E|-|V|)$ elements from $D$. The product of a fixed set of $(|E|-|V|)$ elements in $D$ is weighted by the determinant of the respective minor, with all other elements of $D$ substituted by zero. The remaining rows in $C$ and columns in $B$ look like an unordered incidence matrix of the graph $\tau$ spanned by the remaining $|V|$ edges. Then the determinant of the minor is $(-1)^{|V|}$ times the square of the determinant of the incidence matrix of $\tau$. %
\end{proof}

\begin{lm}\label{multiplicity}
The total multiplicities of the eigenvalues $\lambda=(2k+1)^2\pi^2+O(1)$ are exactly $|E|-|V|$.
\end{lm}
\begin{proof}
If $q=0$, $\det M$ is a polynomial of $\cos{\sqrt{\lambda}}$ and $\sin{\sqrt{\lambda}}$, hence its zeros are $2\pi$-periodic in $\sqrt{\lambda}$. ${\lambda}=(2k+1)^2\pi^2+O(1)$ with periodicity implies $\sqrt{\lambda}=(2k+1)\pi$ with finitely many exceptions. For $\lambda=(2k+1)^2\pi^2$ $A$ and $D$ are zero matrices, thus the rank of $M$ is $2|V|$, and its nullspace is exactly $(|E|-|V|)$-dimensional. As $|\lambda_n(q)-\lambda_n(0)|=O(\|q\|_1)$ (\cite{Kiss2016Ambarzumian}), the total multiplicity of eigenvalues $\lambda=(2k+1)^2\pi^2+O(1)$ is the same for all $q\in L^1$.
\end{proof}

\begin{cor}
If $\lambda=(2k+1)^2\pi^2+O(1)$ and the graph $G$ consists of $r$ odd cycles glued together at a common vertex, then
\begin{align}\label{detMforspecialG}
\det M=-4\sum_{i=1}^{r}\prod_{j\ne i}\sum_{l=1}^{n_j}\sqrt{\lambda}s_{jl}(1,\lambda)+o(\lambda^{-\frac{r}{2}}).
\end{align}
\end{cor}
\begin{proof}
Odd saturated forests of $G$ are connected, hence they are given by leaving out one edge from every but one cycle. Note also that $|V|$ is odd and $|E|-|V|=r-1$. Then the statement follows from (\ref{detM}).
\end{proof}

Substituting the asymptotics (\ref{asym1})-(\ref{asym4}) we get
\begin{cor}
If $\lambda=(2k+1)^2\pi^2+d+o(1)$ and the graph $G$ consists of $r$ odd cycles glued together at a common vertex, then
\begin{align}
\det M=-4\left(\frac{-1}{2\sqrt{\lambda}}\right)^{r-1}p(d)+o(\lambda^{-\frac{r-1}{2}}),
\end{align}
where
\begin{align}\label{eigpol}
p(d)=\sum_{i=1}^{r}\prod_{j\ne i}\left(n_jd-\sum_{l=1}^{n_j}\int_0^{1}q_{jl}\right).
\end{align}
\end{cor}

\begin{lm}\label{pd}
$p(d)=\sum_{i=1}^{r}\prod_{j\ne i}n_jd^{r-1}$.
\end{lm}
\begin{proof}
$\lambda$ is an eigenvalue of the eigenvalue problem (\ref{sch})-(\ref{Kirchhoff}) if and only if $\det M(\lambda)=0$. Let the distinct roots of $p(d)$ be $d_1,\ldots,d_l$. By the previous corollary for $\lambda=(2k+1)^2\pi^2+O(1)$ the distinct roots of $\det M(\lambda)$ are exactly of the form $\lambda=(2k+1)^2\pi^2+d_j+o(1)$ $(1\le j\le l)$. Comparing this with Lemma \ref{multiplicity}, $d_j=0$, hence $p(d)=cd^r$. The principal coefficient is given by (\ref{eigpol}).
\end{proof}

\noindent{\bf Proof of Theorem \ref{oddcycles}.}

Let us introduce $Q_j=\sum_{l=1}^{n_j}\int_0^{1}q_{jl}$. 
For a fixed $m$ substituting $d=\frac{Q_m}{n_m}$ to (\ref{eigpol}), we get
\begin{align}
\frac{1}{\prod_{j=1}^{r}n_j}p(\frac{Q_m}{n_m})=\frac{1}{n_m}\prod_{j\ne m}\left(\frac{Q_m}{n_m}-\frac{Q_j}{n_j}\right)=\sum_{i=1}^{r}\frac{1}{n_i}\left(\frac{Q_m}{n_m}\right)^{r-1}.
\end{align}
Introducing $h_j=\frac{Q_j}{n_j}$, %
\begin{align}
\frac{1}{n_m}\prod_{j\ne m}\left(h_m-h_j\right)=\sum_{i=1}^{r}\frac{1}{n_i}h_m^{r-1}\quad(m=1,2,\ldots,r).
\end{align}
We can assume $h_1\ge h_2\ge\ldots\ge h_r$. Then for $m=2$ the left hand side is nonpositive, hence $h_2\le0$. Similarly, $h_{r-1}\ge0$. Hence for $m=1$,
\begin{align}
\frac{1}{n_1}h_1^{r-2}\left(h_1-h_r\right)=\sum_{i=1}^{r}\frac{1}{n_i}h_1^{r-1}.
\end{align}
If $h_1\ne0$, $\frac{1}{n_1}\left(h_1-h_r\right)=\sum_{i=1}^{r}\frac{1}{n_i}h_1$. Similarly, if $h_r\ne0$, $\frac{1}{n_r}\left(h_r-h_1\right)=\sum_{i=1}^{r}\frac{1}{n_i}h_r$. Subtracting,
\begin{align}
\sum_{i=2}^{r-1}\frac{1}{n_i}(h_1-h_r)=0.
\end{align}
As $n_j$'s are positive, if $r>2$ then $h_1=h_2=\ldots=h_r=0$, while for $r=2$, $h_1n_1+h_2n_2=0$. In both cases,  
\begin{align}
\sum_jQ_j=\sum_{e_{jl}\in G}\int_0^{1}q_{jl}=\int_Gq=0. 
\end{align}
Let us denote the operator of the eigenvalue problem (\ref{sch})-(\ref{Kirchhoff}) by $L$. $\frac{\langle\varphi,L\varphi\rangle}{\langle\varphi,\varphi\rangle}\ge\lambda_0=0$ and equality holds if and only if $\varphi$ is an eigenfunction of $L$. It follows that the constant $1$ must be an eigenfunction corresponding to the eigenvalue $0$. Substituting this to (\ref{sch}) gives $q(x)=0$.
\qed

\end{section}
\begin{section}{Glossary}

A \emph{walk} $W$ in a graph is an alternating sequence of vertices and edges, say $X_0,e_1,\ldots,e_l,X_l$ where $e_i = X_{i-1} X_i$, ($0<i<l$). The length of $W$ is $l$. This walk $W$ is called a \emph{trail} if all its edges are distinct. A \emph{path} is a walk with distinct vertices. A trail whose end vertices coincide (a closed trail) is called a \emph{circuit}. To be precise, a circuit is a closed trail without distinguished endvertices and direction, so that, for example, two triangles sharing a single vertex give rise to precisely two circuits with six edges. If a walk $W = X_0,e_1,\ldots,e_l,X_l$ is such that $l > 3$, $X_0 = X_l$, and the vertices $X_i, 0 < i < l$, are distinct from each other and $X_0$, then $W$ is said to be a \emph{cycle} (\cite{Bollobas1998Modern}, p.5).

The \emph{incidence matrix} of a graph has a row for each vertex and a column for each edge, and is defined as
\begin{align}
R=(r_{ij}), \quad r_{ij}=\left\{\begin{array}{ccc}0&\textrm{ if $e_j$ is not incident to $v_i$,}\\1&\textrm{ if $e_j$ is not a loop and incident to $v_i$,}\\2&\textrm{ if $e_j$ is a loop at $v_i$.}\end{array}\right.
\end{align}
\end{section}

\bibliographystyle{amsplain}

\begin{thebibliography}{10}

\bibitem{AkkermansComtetDesboisMontambauxTexier2000}
Eric Akkermans, Alain Comtet, Jean Desbois, Gilles Montambaux, and Christophe
  Texier, \emph{Spectral determinant on quantum graphs}, Annals of Physics
  \textbf{284} (2000), no.~1, 10--51.

\bibitem{A}
V.~Ambarzumian, \emph{{\"Uber eine Frage der Eigenwerttheorie}}, Zeitschrift
  f\"ur Physik \textbf{53} (1929), 690--695.

\bibitem{BerkolaikoKuchment2013book}
G~Berkolaiko and P~Kuchment, \emph{Introduction to quantum graphs},
  Mathematical Surveys and Monographs, AMS, 2013.

\bibitem{Bollobas1998Modern}
B\'ela Bollob\'as, \emph{Modern graph theory}, Springer, New York, 1998.

\bibitem{B1}
G.~Borg, \emph{{Eine Umkehrung der Sturm-Liouvilleschen Eigenwertaufgabe}},
  Acta Math. \textbf{78} (1946), 1--96.

\bibitem{CarlsonPivovarchik2007ambarzumian}
Robert Carlson and Vyacheslav Pivovarchik, \emph{Ambarzumian's theorem for
  trees}, Electronic Journal of Differential Equations \textbf{2007} (2007),
  no.~142, 1--9.

\bibitem{CarlsonPivovarchik2008SpectralAsymptotics}
\bysame, \emph{Spectral asymptotics for quantum graphs with equal edge
  lengths}, Journal of Physics A: Mathematical and Theoretical \textbf{41}
  (2008), no.~14, 145202.

\bibitem{ChengWangWu2010}
YH~Cheng, Tui-En Wang, and Chun-Jen Wu, \emph{{A note on eigenvalue asymptotics
  for Hill's equation}}, Applied Mathematics Letters \textbf{23} (2010), no.~9,
  1013--1015.

\bibitem{CurrieWatson2005}
Sonja Currie and Bruce~A Watson, \emph{Eigenvalue asymptotics for differential
  operators on graphs}, Journal of computational and applied mathematics
  \textbf{182} (2005), no.~1, 13--31.

\bibitem{Davies2013}
E.B Davies, \emph{An inverse spectral theorem}, Journal of Operator Theory
  \textbf{69} (2013), no.~1, 195--208.

\bibitem{Desbois2000}
Jean Desbois, \emph{{Spectral determinant of Schr\"odinger operators on
  graphs}}, Journal of Physics A: Mathematical and General \textbf{33} (2000),
  no.~7, L63.

\bibitem{Friedlander2006}
Leonid Friedlander, \emph{{Determinant of the Schr\"odinger Operator on a
  Metric Graph}}, Contemporary Mathematics \textbf{415} (2006), 151--160.

\bibitem{HarrisonKirstenTexier2012}
J~M Harrison, K~Kirsten, and C~Texier, \emph{{Spectral determinants and zeta
  functions of Schr\"odinger operators on metric graphs}}, Journal of Physics
  A: Mathematical and Theoretical \textbf{45} (2012), no.~12, 125206.

\bibitem{Horvath2005Annals}
M.~Horv\'ath, \emph{{Inverse spectral problems and closed exponential
  systems}}, Ann. of Math. \textbf{162} (2005), no.~2, 885--918.

\bibitem{Horvath2015}
\bysame, \emph{{On the stability in Ambarzumian theorems}}, Inverse Problems
  \textbf{31} (2015), no.~2, 025008.

\bibitem{KacPivovarchik2011}
I~Kac and V~Pivovarchik, \emph{On multiplicity of a quantum graph spectrum},
  Journal of Physics A: Mathematical and Theoretical \textbf{44} (2011),
  no.~10, 105301.

\bibitem{Kiss2016Ambarzumian}
M{\'a}rton Kiss, \emph{{Spectral determinants and an Ambarzumian theorem on
  graphs}}, arXiv preprint arXiv:1610.00971v2 (2016).

\bibitem{Kuchment2008}
Pavel Kuchment, \emph{Quantum graphs: an introduction and a brief survey}, in:
  Analysis on Graphs and Its Applications (2008), 291--314.

\bibitem{LawYanagida2012}
Chun-Kong Law and Eiji Yanagida, \emph{{A solution to an Ambarzumyan problem on
  trees}}, Kodai Mathematical Journal \textbf{35} (2012), no.~2, 358--373.

\bibitem{MollerPivovarchik2015book}
Manfred M{\"o}ller and Vyacheslav Pivovarchik, \emph{{Spectral Theory of
  Operator Pencils, Hermite-Biehler Functions, and their Applications}},
  Operator Theory: Advances and Applications, Springer, 2015.

\bibitem{Pankrashkin2006}
Konstantin Pankrashkin, \emph{{Spectra of Schr{\"o}dinger operators on
  equilateral quantum graphs}}, Letters in Mathematical Physics \textbf{77}
  (2006), no.~2, 139--154.

\bibitem{Pivovarchik2005ambarzumian}
VN~Pivovarchik, \emph{{Ambarzumian's Theorem for a Sturm-Liouville boundary
  value problem on a star-shaped graph}}, Functional Analysis and Its
  Applications \textbf{39} (2005), no.~2, 148--151.

\bibitem{PokornyiBorovskikh2004}
Yu~V Pokornyi and AV~Borovskikh, \emph{Differential equations on networks
  (geometric graphs)}, Journal of Mathematical Sciences \textbf{119} (2004),
  no.~6, 691--718.

\bibitem{Texier2010}
Christophe Texier, \emph{$\zeta$-regularized spectral determinants on metric
  graphs}, Journal of Physics A: Mathematical and Theoretical \textbf{43}
  (2010), no.~42, 425203.

\bibitem{Yang20161348}
Chuan-Fu Yang and Xiao-Chuan Xu, \emph{Ambarzumyan-type theorems on graphs with
  loops and double edges}, Journal of Mathematical Analysis and Applications
  \textbf{444} (2016), no.~2, 1348 -- 1358.

\end{thebibliography}
\providecommand{\bysame}{\leavevmode\hbox to3em{\hrulefill}\thinspace}
\providecommand{\MR}{\relax\ifhmode\unskip\space\fi MR }
\providecommand{\MRhref}[2]{%
  \href{http://www.ams.org/mathscinet-getitem?mr=#1}{#2}
}
\providecommand{\href}[2]{#2}

\end{document}